\documentclass[a4paper,10pt]{article}

\usepackage[utf8]{inputenc}
\usepackage[T1]{fontenc}
\usepackage{amsmath,amssymb}
\usepackage{amsthm}
\usepackage[scale=0.8]{geometry}
\usepackage[english]{babel}
\usepackage{enumerate}
\usepackage{float}
\usepackage{url}

\DeclareMathOperator{\Gal}{Gal}
\DeclareMathOperator{\im}{Im}
\DeclareMathOperator{\SL}{SL}
\DeclareMathOperator{\pgcd}{gcd}

\renewcommand{\bar}{\overline}
\renewcommand{\mapsto}{\longmapsto}

\newcommand{\C}{\mathbb C}
\newcommand{\HH}{\mathbb H}
\newcommand{\Q}{\mathbb Q}
\newcommand{\Z}{\mathbb Z}

\newcommand{\OO}{\mathcal O}

\theoremstyle{plain}

\newtheorem{theoreme}{Theorem}[section]
\newtheorem{lemme}[theoreme]{Lemma}
\newtheorem{proposition}[theoreme]{Proposition}

\newtheorem{conjecture}[theoreme]{Conjecture}

\theoremstyle{definition}

\theoremstyle{remark}

\newtheorem{remarque}[theoreme]{Remark}

\title{Equations with powers of singular moduli}
\author{Antonin Riffaut}
\date{\today}

\begin{document}
%===============

\maketitle

\begin{abstract}

We treat two different equations involving powers of singular moduli. On the one hand, we show that, with two possible (explicitly specified) exceptions, two distinct singular moduli $j(\tau),j(\tau')$ such that the numbers $1$, $j(\tau)^m$ and $j(\tau')^n$ are linearly dependent over $\mathbb{Q}$ for some positive integers $m,n$, must be of degree at most $2$. This partially generalizes a result of Allombert, Bilu and Pizarro-Madariaga, who studied CM-points belonging to straight lines in $\mathbb{C}^2$ defined over $\mathbb{Q}$. On the other hand, we show that, with “obvious” exceptions, the product of any two powers of singular moduli cannot be a non-zero rational number. This generalizes a result of Bilu, Luca and Pizarro-Madariaga, who studied CM-points belonging to an hyperbola $xy=A$, where $A\in\mathbb{Q}$.
 
\end{abstract}

\section{Introduction}

Let~$j$ be the classical $j$-function on the Poincaré plane ${\HH=\{z\in \C:\im z>0\}}$. 
A \textsl{singular modulus} %is the $j$-invariant of an elliptic curve with a complex multiplication; in other words, it
 is a number of the form $j(\tau)$, where ${\tau\in \HH}$ is a complex algebraic number of degree~$2$. % with positive imaginary part. 
 It is known that $j(\tau)$ is an algebraic integer and Class Field Theory tells that
\[[\Q(j(\tau)):\Q]=[\Q(\tau,j(\tau)):\Q(\tau)]=h_\Delta,\]
the class number of the order ${\OO_\Delta=\Z[(\Delta+\sqrt\Delta)/2]}$, where~$\Delta$ is the discriminant of the minimal polynomial of~$\tau$ over~$\Z$. Moreover, $\Q(\tau,j(\tau))/\Q(\tau)$ is an abelian Galois extension with Galois group (canonically) isomorphic to the class group of the order $\OO_\Delta$. One can also interpret $\OO_\Delta$ as the automorphism ring of the lattice $\langle1,\tau\rangle$, or of the corresponding elliptic curve. For all details see, for instance, \cite[§7 and §11]{Co89}.

During the last decades, motivated by the celebrated André-Oort conjecture, many researchers studied Diophantine properties of the singular moduli. The starting point was the following result of André~\cite{An98}.

\begin{theoreme}[André]

\label{than}
Let ${f(X,Y)\in \bar\Q[X,Y]}$ be an irreducible polynomial with 
\[\deg_{X}f,\,\deg_{Y}f>0.\]
Assume that, for every positive integer~$N$, the polynomial $f(X,Y)$ is not proportional to the classical modular polynomial $\Phi_N(X,Y)$ \cite[§11C]{Co89}. Then the equation ${f(x,y)=0}$ has at most finitely many solutions in singular moduli $x,y$. 

\end{theoreme}

Independently, the same result was obtained by Edixhoven~\cite{Ed98}, but his proof was conditional on the GRH. In 2009,  Pila~\cite{Pi09} gave another proof of André's theorem, based on an idea of Pila and Zannier~\cite{Pi08}. % and later~\cite{Pi11} extended it to higher dimension. 
However, both André's and Pila's proofs were non-effective, because they use the Siegel-Brauer lower estimate for the class number. 

At present, much more general results on the André-Oort conjecture are available, due to the recent spectacular work of Klingler, Pila, Tsimerman, Ullmo and Yafaev, see \cite{Kl14,Pi11,PT13,UY14} and the references therein. However, like the initial arguments of André and Edixhoven, these results are either ineffective, or conditional on various forms of the GRH. 

Unconditional effective proofs of André's theorem were discovered only in 2012, by Kühne~\cite{Ku12,Ku13} and, independently, by Bilu, Masser and Zannier~\cite{Yu13}. They showed that, in the set-up of Theorem~\ref{than}, the discriminants $\Delta_x,\Delta_y$ of the singular moduli $x,y$ with ${f(x,y)=0}$ satisfy ${|\Delta_x|,|\Delta_y|\le c(f)}$, where $c(f)$ can be effectively computed in terms of the polynomial~$f$.  

This opened the possibility of not only proving finiteness results for  equations involving singular moduli, but   solving some of these equations completely. For instance, Kühne~\cite{Ku13} proved that equation ${x+y=1} $ has no solutions in singular moduli, and in~\cite{Yu13} the same result was obtained for the equation ${xy=1}$. 

All these results were substantially generalized in two recent articles \cite{Al15,Yu15}. In~\cite{Al15}  solutions of all linear equations ${Ax+By=C}$ with ${A,B,C\in \Q}$ were determined. In~\cite{Yu15} the same result was obtained for equations ${xy=A}$. Note that in the latter case an infinite one-parametric family of equations was completely solved in singular moduli, and in the former case a two-parametric family is completely solved. Here are the principal results of \cite{Al15,Yu15}.

\begin{theoreme}[Allombert et al.~\cite{Al15}]
\label{thlines}
Let $A,B,C$ be rational numbers, ${AB\ne 0}$, and $x,y$ singular moduli. Assume that ${Ax+By=C}$. Then we have one of the following options:

\begin{description}

\item[(the trivial case)]
${A+B=C=0}$ and ${x=y}$;

\item[(the rational case)]
${x,y\in \Q}$;

\item[(the quadratic case)]
${x\ne y}$ and ${\Q(x)=\Q(y)}$ is a number field of degree~$2$. 

\end{description}

\end{theoreme}

\begin{theoreme}[Bilu et al.~\cite{Yu15}]
\label{thhyp}

Let $x,y$ be singular moduli such that ${xy\in \Q^\times}$. Then we have one of the following options:

\begin{description}

\item[(the rational case)]
${x,y\in \Q^\times}$;

\item[(the quadratic case)]
$x$ and $y$ are of degree~$2$ and conjugate over~$\Q$. 
\end{description}

\end{theoreme}

Two observations can be made here. First of all, both results are best possible. For instance, in both the rational and the quadratic case of Theorem~\ref{thlines} one easily finds  ${A,B,C \in \Q}$ such that ${AB\ne 0}$ and ${Ax+By=C}$. Similarly, in both rational and quadratic cases of Theorem~\ref{thhyp} one has ${xy\in \Q^\times}$. 

Second, the lists of singular moduli of  degrees~$1$ and~$2$  over~$\Q$ are widely available or can be easily generated using a suitable computer package, like \textsf{PARI}~\cite{Pari}. In particular, there are $13$ rational singular moduli, and $29$ pairs of $\Q$-conjugate singular moduli of degree~$2$; see \cite[Section~1]{Yu15} for more details. This means that Theorems~\ref{thlines} and~\ref{thhyp} give a completely explicit characterization of all solutions. 

The aim of the present article is to generalize  Theorems~\ref{thlines} and~\ref{thhyp}, by introducing exponents; that is, instead of equations ${Ax+By=C}$ and ${xy=A}$, to consider more general equations ${Ax^m+By^n=C}$ and ${x^my^n=A}$, where the exponents~$m$ and~$n$ are unknown as well. We propose the following conjecture.

\begin{conjecture}

Let $A,B,C$ be rational numbers with ${AB\ne 0}$, let $x,y$ be singular moduli and let $m,n$ be positive integers. Assume that ${Ax^m+By^n=C}$. Then we have one of the following options:

\begin{description}

\item[(the trivial case)]
${A+B=C=0}$, ${x=y}$ and ${m=n}$;

\item[(the rational case)]
${x,y\in \Q}$;

\item[(the quadratic case)]
${x\ne y}$ and ${\Q(x)=\Q(y)}$ is a number field of degree~$2$. 

\end{description}

\end{conjecture}

The result we obtain here is only slightly weaker than this conjecture. 

\begin{theoreme}
\label{th:main}

Let ${x=j(\tau)}$ and ${y=j(\tau')}$ be two distinct singular moduli of respective discriminants $\Delta$ and $\Delta'$, and $m,n$ two positive integers. Assume that  $Ax^m+By^n=C$, for some $A,B,C\in\mathbb{Q}^{\times}$. Then $x$ and $y$ generate the same number field over $\mathbb{Q}$ of degree $h\leq 3$. Moreover, if $h=3$, then either $\{\Delta,\Delta'\}=\{-23,-4\cdot 23\}$ or $\{\Delta,\Delta'\}=\{-31,-4\cdot 31\}$.

\end{theoreme}

There are two cases that our current methods are not able to handle. First of all, it is the case ${x=y}$, which is equivalent to the following question: can a singular modulus of degree~$3$ or higher be a root of a trinomial with rational coefficients? Much about trinomials is known, but this knowledge is still insufficient to rule out such a possibility. Another case we cannot rule out is that of distinct~$x$ and~$y$ of degree~$3$. There are only two such pairs, as indicated in the statement of the theorem. However, in a more recent work, Luca and Riffaut \cite{Lu17} eliminated these two remaining pairs. Consequently, Theorem \ref{th:main} together with \cite[Theorem 1.2]{Lu17} completely solve the above equation for distinct singular moduli, and we deduce the following Theorem.

\begin{theoreme}
Let $x=j(\tau),y=j(\tau')$ be two distinct singular moduli of respective discriminants $\Delta$ and $\Delta'$, and $m,n$ two positive integers. Assume that $Ax^m+By^n=C$, for some $A,B,C\in\mathbb{Q}^{\times}$. Then $x$ and $y$ generate the same number field over $\mathbb{Q}$ of degree at most $2$.
\end{theoreme}

The assumption ${C\ne 0}$ is seemingly restrictive, but, in fact, the case ${C=0}$ is contained in Theorem~\ref{th:main-2} below. 

For the equation ${x^my^n=A}$, where ${A\in \Q^\ast}$, we do not state any conjecture, because the result we obtain is  best possible.  Moreover, this time we allow the exponents~$m$ and~$n$ to be arbitrary non-zero integers, positive or negative. 

\begin{theoreme}

\label{th:main-2}
Let $x,y$ be two non-zero singular moduli, and $m,n$ two non-zero integers. Assume that ${x^my^n\in\mathbb{Q}^{\times}}$. Then we have one of the following options:

\begin{description}

\item[(the equality case)] $x=y$ and $m+n=0$.

\item[(the rational case)] ${x,y\in \Q^\ast}$;

\item[(the quadratic case)] $m=n$ and $x,y$ are of degree $2$ and conjugate over $\mathbb{Q}$.

\end{description}

\end{theoreme}

It is interesting to compare this theorem with a recent result of Pila and Tsimerman~\cite{PT17}, who proved that, given a positive integer~$k$, there exist at most finitely many $k$-tuples of pairwise distinct multiplicatively dependent singular moduli. This result is, however, non-effective.

Our calculations were performed using the \textsf{PARI/GP} package \cite{Pari}. The sources are available from the author.

\section{Preliminaries and main lemmas}

\subsection{Estimations of the $j$-invariant function}

We denote by $\mathcal{D}$ the standard fundamental domain of the action of $\SL(2,\mathbb{Z})$ on $\mathbb{H}$, that is the closed hyperbolic triangle with vertices
\[\zeta_3=\frac{-1+\sqrt{-3}}{2},\quad\zeta_6=\frac{1+\sqrt{3}}{2},\quad\infty.\]
For $\tau\in\mathbb{H}$, we denote $q(\tau)=\mathrm{e}^{2i\pi\tau}$, and we will simply write $q=q(\tau)$ when there is no ambiguity.

When $j(\tau)$ is large, it is approximatively of the same magnitude as $q^{-1}$. The following is Lemma 1 from \cite{Yu13}, which makes this explicit.

\begin{lemme}
\label{lemme:estim-j}

For $\tau\in\mathcal{D}$, we have
\[\left||j(\tau)|-|q|^{-1}\right|\leq 2079.\]

\end{lemme}

This estimate has the following consequence.

\begin{lemme}
\label{lemme:estim-j-bis}

For $\tau\in\mathcal{D}\setminus\{\mathrm{e}^{i\pi/3},\mathrm{e}^{2i\pi/3}\}$, we have
\[|j(\tau)|=|q|^{-1}\mathrm{e}^{v(q)},\]
with $v(q)$ a real number satisfying $|v(q)|\leq 2883|q|$ as soon as $\im\tau\geq\log 4158/2\pi\approx 1.326$. Furthermore, if $\im\tau\geq\log 5766/2\pi\approx 1.378$, then $|v(q)|\leq 1/2$.

\end{lemme}

\begin{proof}

From Lemma \ref{lemme:estim-j}, we have
\[|q|^{-1}(1-2079|q|)\leq|j(\tau)|\leq|q|^{-1}(1+2079|q|),\]
then
\[|j(\tau)|=|q|^{-1}(1+u(q)),\]
with $|u(q)|\leq 2079|q|$.
If $\im\tau\geq\log 4158/2\pi$, then $|u(q)|\leq 1/2$. As $|\log(1+x)|\leq 2\log 2|x|$, for all $x\in\,]-1/2,1/2[$, we deduce that
\[|v(q)|=|\log(1+u(q))|\leq 2\log 2|u(q)|\leq 2883|q|. \qedhere\]

\end{proof}

Finally, we need a bound on the logarithm of $|j(\tau)|$.

\begin{lemme}
\label{lemme:estim-logj}

For $\tau\in\mathcal{D}$, we have
\[\log|j(\tau)|\leq 9\im\tau.\]

\end{lemme}

\begin{proof}

First, using again Lemma \ref{lemme:estim-j}, we have
\[\log|j(\tau)|\leq\log\left(2079+|q|^{-1}\right).\]
Since the function $x\mapsto\log(2079+x)/\log x$ is decreasing on $]1,+\infty[$, and $|q|^{-1}=\mathrm{e}^{2\pi\im\tau}\geq\mathrm{e}^{\pi\sqrt{3}}$, we get
\[\log\left(2079+|q|^{-1}\right)\leq\log\left(|q|^{-1}\right)\frac{\log\left(2079+\mathrm{e}^{\pi\sqrt{3}}\right)}{\pi\sqrt{3}}\leq 9\im\tau. \qedhere\]

\end{proof}

\subsection{The conjugates of $j(\tau)$}

Let $j(\tau)$ be a singular modulus of discriminant $\Delta$ and degree $h$ over $\mathbb{Q}$. It is well-known that the conjugates of $j(\tau)$ over $\mathbb{Q}$ can be described explicitly. We briefly recall this description.

Denote by $T_{\Delta}$ the set of triples of integers $(a,b,c)$ such that
\[
\begin{cases}
\pgcd(a,b,c)=1,\\
\Delta=b^2-4ac,\\
-a<b\leq a<c\quad\text{or}\quad 0\leq b\leq a=c.
\end{cases}
\]

\begin{proposition}
\label{prop:conj-j}

The conjugates of $j(\tau)$ over $\mathbb{Q}$ are given by
\[j\left(\frac{-b+\sqrt{\Delta}}{2a}\right),\quad(a,b,c)\in T_{\Delta}.\]
In particular, $|T_{\Delta}|=h$.

\end{proposition}
For a proof, see, for instance, \cite[Proposition 2.5]{Yu15}, which uses results from \cite{Co89}.

In the set $T_{\Delta}$, there exists exactly one triple $(a,b,c)$ with $a=1$, given explicitly by
\[\left(1,r_4(\Delta),\frac{r_4(\Delta)-\Delta}{4}\right),\]
where $r_4(\Delta)\in\{0,1\}$ is defined by $\Delta\equiv r_4(\Delta)\bmod 4$. The corresponding conjugate of $j(\tau)$,
\[j\left(\frac{-r_4(\Delta)+\sqrt{\Delta}}{2}\right),\]
will be called the \emph{dominant $j$-value} of discriminant $\Delta$, and has the property that it is much larger in absolute value than all its other conjugates.

\begin{lemme}
\label{lemme:j-dom}

Let $x_0$ be the dominant $j$-value of discriminant $\Delta$, with $|\Delta|\geq 11$. Let $x$ be a conjugate of $x_0$ over $\mathbb{Q}$, $x\neq x_0$. Then $|x|\leq 0.1|x_0|$.

\end{lemme}
It is a consequence of Lemma \ref{lemme:estim-j}. For a complete proof, we refer to \cite[Lemma 3.5]{Yu13}.

Using this property of the dominant $j$-value, we can prove that the field generated over $\mathbb{Q}$ by a singular modulus is the same as the one generated by any of its powers.

\begin{lemme}
\label{lemme:power-j}

Let $x$ be a singular modulus of discriminant $\Delta$, with $|\Delta|\geq 11$, and $n$ a non-zero integer. Then $\mathbb{Q}(x)=\mathbb{Q}(x^n)$.

\end{lemme}

\begin{proof}

Since $\mathbb{Q}(x^n)=\mathbb{Q}(x^{-n})$, it is sufficient to prove the lemma for $n\geq 2$. Assume that $\mathbb{Q}(x)\neq\mathbb{Q}(x^n)$, which means that $\mathbb{Q}(x^n)$ is a strict subfield of $\mathbb{Q}(x)$. Let $L$ be the Galois closure of the extension $\mathbb{Q}(x)/\mathbb{Q}$. By the Galois correspondence Theorem, the Galois group $\Gal(L/\mathbb{Q}(x))$ is a strict subgroup of $\Gal(L/\mathbb{Q}(x^n))$. Let $\sigma$ be any element of $\Gal(L/\mathbb{Q}(x^n))\setminus\Gal(L/\mathbb{Q}(x))$. We have $\sigma(x)\neq x$ and $\sigma(x^n)=\sigma(x)^n=x^n$. Now, choose $\sigma_0\in\Gal(L/\mathbb{Q})$ such that $\sigma_0(x)$ is the dominant $j$-value of discriminant $\Delta$. We then have $(\sigma_0\sigma)(x)\neq\sigma_0(x)$ and $(\sigma_0\sigma)(x)^n=\sigma_0(x)^n$. This is in contradiction with Lemma \ref{lemme:j-dom}, which claims that $|(\sigma_0\sigma)(x)|<|\sigma_0(x)|$, so that $|(\sigma_0\sigma)(x)^n|<|\sigma_0(x)^n|$. \qedhere

\end{proof}

At present, we would like to describe more precisely the set $T_{\Delta}$ in general, i.e. finding explicit triples $(a,b,c)$ of $T_{\Delta}$ for some small values of $a$. It turns out that, when $\Delta\equiv 1\bmod 8$, then for any power of $2$, say $2^r$ with $r\geq 1$, if $\Delta$ is large enough, there exists a triple $(a,b,c)$ in $T_{\Delta}$ with $a=2^r$. It is a mere consequence of Hensel's Lemma, according to which $\Delta$ is a square in $\mathbb{Z}_2$. As we are only interested in the small values of $a$, we claim the following result.

\begin{lemme}
\label{lemme:T_Delta}

Assume that $\Delta\equiv 1\bmod 8$. Then the set $T_{\Delta}$ contains exactly:

\begin{itemize}

\item two triples $(a,b,c)$ with $a=2$ when $|\Delta|\geq 23$;

\item two triples $(a,b,c)$ with $a=4$ when $|\Delta|\geq 71$;

\item two triples $(a,b,c)$ with $a=8$ when $|\Delta|\geq 239$.

\end{itemize}

\end{lemme}

\begin{proof}

First, notice that, if $(a,b,c)\in T_{\Delta}$, then $b$ must be a solution of the equation $\Delta\equiv b^2\bmod 4a$ satisfying $|b|\leq a$. Therefore, for a fixed value of $a$, solving this equation gives all possible triples $(a,b,c)$ with the specified value of $a$, provided that they verify $\pgcd(a,b,c)=1$, and either $-a<b\leq a\leq c$ or $0\leq b\leq a=c$.

By assumption, $\Delta\equiv 1\bmod 8$. The unique solutions $b\in\mathbb{Z}$ of the equation $b^2\equiv 1\bmod 8$ satisfying $|b|\leq 2$ are $b=1$ and $b=-1$. Let $c=(\Delta-1)/8\in\mathbb{N}$, so that $\Delta=b^2-8c$. If $|\Delta|\geq 23$, then $c\geq 3$. Thus, the two triples $(2,-1,c)$ and $(2,1,c)$ belong to $T_{\Delta}$, and they are the only ones whose first component is equal to $2$.

Next, we have either $\Delta\equiv 1\bmod 16$ or $\Delta\equiv 9\bmod 16$.
\begin{itemize}

\item If $\Delta\equiv 1\bmod 16$, the unique solutions $b\in\mathbb{Z}$ of the equation $b^2\equiv 1\bmod 16$ satisfying $|b|\leq 4$ are still $b=1$ and $b=-1$. Let $c=(\Delta-1)/16\in\mathbb{N}$, so that $\Delta=b^2-16c$. If $|\Delta|\geq 79$, then $c\geq 5$.

\item If $\Delta\equiv 9\bmod 16$, the unique solutions $b\in\mathbb{Z}$ of the equation $b^2\equiv 9\bmod 16$ satisfying $|b|\leq 4$ are $b=3$ and $b=-3$. Let $c=(\Delta-9)/16\in\mathbb{N}$, so that $\Delta=b^2-16c$. If $|\Delta|\geq 71$, then $c\geq 5$.

\end{itemize}
Thus, if $|\Delta|\geq 71$, there are exactly two triples $(a,b,c)$ in $T_{\Delta}$ with $a=4$.

Proving the last assertion of the lemma is identical to the previous reasoning. \qedhere

\end{proof}

We shall also recall that the set $T_{\Delta}$ contains at most two triples $(a,b,c)$ with $a=2$ in general, and in particular no triple when $\Delta\equiv 4\bmod 16$. See \cite[Proposition 2.6]{Yu15} for more details.

\subsection{The modular curve $Y_0(N)$ and the special case $N=2$}

Let $N$ be a positive integer. The \emph{classical modular curve} of level $N$, denoted by $Y_0(N)$, is the algebraic curve corresponding to the quotient of $\mathbb{H}$ by the action of the congruence subgroup
\[\Gamma_0(N)=\left\{\left(
\begin{matrix}
a & b\\
c & d
\end{matrix}
\right)\in\SL(2,\mathbb{Z})\,;\,c\equiv 0\bmod N\right\}\]
of $\SL(2,\mathbb{Z})$. It can be realized as an algebraic curve in $\mathbb{C}^2$ of equation
\[\Phi_N(x,y)=0,\]
where $\Phi_N(X,Y)\in\mathbb{C}[X,Y]$ is called the \emph{classical modular polynomial} of level $N$.

For our purpose, we need a more precise description of the polynomial $\Phi_N$, especially for the case $N=2$. We consider the following set of matrices
\[C(N)=\left\{\left(
\begin{matrix}
a & b\\
0 & d
\end{matrix}
\right)\,;\,ad=N,\,a>0,\,0\leq b<d,\,\pgcd(a,b,d)=1\right\}.\]
For
\[\sigma=\left(
\begin{matrix}
a & b\\
0 & d
\end{matrix}
\right)\in C(N)\]
and $\tau\in\mathbb{H}$, we write
\[\sigma\tau=\frac{a\tau+b}{d}\in\mathbb{H}.\]
For all $\tau\in\mathbb{H}$, we have
\[\Phi_N(X,j(\tau))=\prod_{\sigma\in C(N)}(X-j(\sigma\tau)).\]
We refer to section 11.B of \cite{Co89} for more details. When $N=2$, the set $C(N)$ can easily be determined, namely
\[C(2)=\left\{
\left(
\begin{matrix}
2 & 0\\
0 & 1
\end{matrix}
\right),\,
\left(
\begin{matrix}
1 & 0\\
0 & 2
\end{matrix}
\right),\,
\left(
\begin{matrix}
1 & 1\\
0 & 2
\end{matrix}
\right)
\right\}.\]
Thus,
\[\Phi_2(X,j(\tau))=(X-j(2\tau))\left(X-j\left(\frac{\tau}{2}\right)\right)\left(X-j\left(\frac{\tau+1}{2}\right)\right).\]

Now, let $(j(\tau),j(\tau'))$ be a CM-point. Then $(j(\tau),j(\tau'))$ lies on the curve $Y_0(2)$ if and only if
\[j(\tau)\in\left\{j(2\tau'),j\left(\frac{\tau'}{2}\right),j\left(\frac{\tau'+1}{2}\right)\right\}.\]
Assume that both $\tau$ and $\tau'$ belong to $\mathcal{D}$.

\begin{itemize}

\item The equality $j(\tau)=j(2\tau')$ holds if and only if there exists $\gamma\in\SL(2,\mathbb{Z})$ such that $j(\tau)=j(\gamma(2\tau'))$. Since $\tau$ belongs to $\mathcal{D}$, and the same goes for $2\tau'$ up to a translation $z\mapsto z+\varepsilon\in\SL(2,\mathbb{Z})$ by some $\varepsilon\in\{0,\pm 1\}$, this condition is equivalent to $\tau=2\tau'+\varepsilon$ for some $\varepsilon\in\{0,\pm 1\}$.

\item Likewise, if $|\tau'|\geq 2$, then $\tau'/2$ belongs to $\mathcal{D}$, else we can transform $\tau'/2$ into an element of $\mathcal{D}$ by applying successively the inversion $z\mapsto-1/z\in\SL(2,\mathbb{Z})$ and a translation $z\mapsto z+\varepsilon\in\SL(2,\mathbb{Z})$ by some $\varepsilon\in\{0,\pm 1\}$. Therefore, the equality $j(\tau)=j(\tau'/2)$ holds if and only if $\tau=\tau'/2$ or $\tau=-2/\tau'+\varepsilon$ for some $\varepsilon\in\{0,\pm 1\}$. By the same way, the equality $j(\tau)=j((\tau'+1)/2)$ holds if and only if $\tau=(\tau'+1)/2$ or $\tau=-1/((\tau'+1)/2-\varepsilon')+\varepsilon$ for some $\varepsilon'\in\{0,1\}$ and $\varepsilon\in\{0,\pm 1\}$.

\end{itemize}

\begin{remarque}
\label{rmq:test-Y02}

When $\tau$ and $\tau'$ both belong to a quadratic number field $K$, this gives us an exact algorithmic way of determining whether $(j(\tau),j(\tau'))$ lies on the curve $Y_0(2)$ or not: we have to check if $\tau=\gamma\tau'$ for one of the transformation $\gamma\in\SL(2,\mathbb{Z})$ mentioned above.

\end{remarque}

It also provides us the following useful Lemma.

\begin{lemme}
\label{lemme:Y02}

Let $(x,y)=(j(\tau),j(\tau'))$ be a CM-point, with both $\tau$ and $\tau'$ belonging to $\mathcal{D}$. Assume that $(x,y)$ lies on the curve $Y_0(2)$ and that $\im\tau'\geq 2$. Then $\im\tau/\im\tau'\in\{2,1/2\}$.

\end{lemme}

\subsection{Height of a singular modulus}

We shall recall some basic facts about the height of an algebraic number. Let $\alpha$ be a non-zero algebraic number of degree $d$ over $\mathbb{Q}$, and $\alpha_1=\alpha,\alpha_2,\dots,\alpha_d$ be all of its conjugates in $\bar{\mathbb{Q}}$. The logarithmic height of $\alpha$, denoted by $\mathrm{h}(\alpha)$, is defined to be
\[\mathrm{h}(\alpha)=\frac{1}{d}\sum_{k=1}^d\log\max\{1,|\alpha_k|\}.\]
Here are some useful properties of the logarithmic height.

\begin{itemize}

\item For any non-zero algebraic number $\alpha$ and $\lambda\in\mathbb{Q}$, we have $\mathrm{h}(\alpha^{\lambda})=|\lambda|\mathrm{h}(\alpha)$. In particular, $\mathrm{h}(1/\alpha)=\mathrm{h}(\alpha)$. See \cite[Lemma 1.5.18]{Bo06}.

\item For any two non-zero algebraic numbers $\alpha$ and $\beta$, we have $\mathrm{h}(\alpha\beta)\leq\mathrm{h}(\alpha)+\mathrm{h}(\beta)$.

\end{itemize}

For our problem, we are interested in bounding the height of a singular modulus. In his article, Kühne gives a bound on the height of a singular modulus depending on its discriminant $\Delta$, its degree $h$ over $\mathbb{Q}$, and some arbitrary parameter $\varepsilon>0$, see \cite[Lemma 3]{Ku13}. We only need the following simpler bound.

\begin{lemme}
\label{lemme:height-j}

Let $x$ be a singular modulus of discriminant $\Delta$. We have
\[\mathrm{h}(x)\leq\frac{9|\Delta|^{1/2}}{2}.\]

\end{lemme}

\begin{proof}

Let $h$ be the degree of $x$ over $\mathbb{Q}$. From Proposition \ref{prop:conj-j}, we can write
\[\mathrm{h}(x)=\frac{1}{h}\sum_{(a,b,c)\in T_{\Delta}}\log\max\left\{1,\left|j\left(\frac{-b+\sqrt{\Delta}}{2a}\right)\right|\right\}.\]
Then, by applying Lemma \ref{lemme:estim-logj}, we have
\[\mathrm{h}(x)\leq\frac{9|\Delta|^{1/2}}{2h}\sum_{(a,b,c)\in T_{\Delta}}\frac{1}{a}\leq\frac{9|\Delta|^{1/2}}{2}. \qedhere\]

\end{proof}

\subsection{Linear forms in the logarithms of algebraic numbers}

We need to estimate linear forms in the logarithms of algebraic numbers, that is expressions of the form
\[\Lambda=b_1\log\alpha_1+\dots+b_r\log\alpha_r,\]
where $b_1,\dots,b_r\in\mathbb{Z}$ and $\alpha_1,\dots,\alpha_r$ are non-zero algebraic numbers. Here, $\log\alpha_1,\dots,\log\alpha_r$ are arbitrary fixed non-zero values of the logarithms.

The article \cite{Ma00} of Matveev gives a bound on the modulus of such an expression. The following Theorem is a particular case and a slightly modified version of \cite[Corollary 2.3]{Ma00}.

\begin{theoreme}
\label{th:matveev}

Assume that $\alpha_1,\dots\alpha_r$ are not all real and $\Lambda\neq 0$. Let $d=[\mathbb{Q}(\alpha_1,\dots,\alpha_r):\mathbb{Q}]$, $H=\max\{|b_1|,\dots,|b_r|\}$, and $A_1,\dots,A_r$ be real numbers such that
\[A_j\geq\max\{\mathrm{h}(\alpha_j),\lvert\log\alpha_j|/d,0.16/d\},\quad j\in\{1,\dots,r\}.\]
Then
\begin{equation}
\label{eq:matveev}
\log|\Lambda|>-2^{6r+20}d^{2+r}A_1\dots A_r\log(\mathrm{e}d)\log(\mathrm{e}H).
\end{equation}

\end{theoreme}

\subsection{Multiplicative independence of algebraic integers}
\label{ssec:mult_ind}

In this part, we develop an algorithmic way of proving that two algebraic numbers are multiplicatively independent. This will be required to complete our proofs of both Theorems \ref{th:main} and \ref{th:main-2}.

Recall that two algebraic numbers $\alpha,\beta$ of a number field $L$ are said to be \emph{multiplicatively independent} if, for all integers $m,n$, the equality $\alpha^m=\beta^n$ implies $m=n=0$. Otherwise, they are said to be \emph{multiplicatively dependent}. Notice that, if $\alpha^m=-\beta^n$ for some integers $m,n$ non-zero simultaneously, then $\alpha$ and $\beta$ are multiplicatively dependent, since $\alpha^{2m}=\beta^{2n}$ in this case.

Assume that $\alpha$ and $\beta$ are multiplicatively dependent, and set $m,n$ two integers non-zero simultaneously such that $\alpha^m=\beta^n$. Assume moreover that $|\alpha|,|\beta|\neq 1$, so that $m,n\neq 0$. Choose a prime ideal $\mathfrak{p}$ of $\mathcal{O}_L$ dividing both fractional ideals $\alpha\mathcal{O}_L$ and $\beta\mathcal{O}_L$. If no such ideal exists, then either we can directly conclude that $\alpha$ and $\beta$ are multiplicatively independent, or they are units. It turns out that this last possibility does not occur in our calculations, but it can still be checked in general by using complex embeddings. Otherwise, we have
\[\frac{m}{n}=\frac{v_{\mathfrak{p}}(\beta\mathcal{O}_L)}{v_{\mathfrak{p}}(\alpha\mathcal{O}_L)},\]
where $v_{\mathfrak{p}}(\cdot)$ is the valuation at the prime ideal $\mathfrak{p}$. This allows to compute a theoretical value $k/l$ under irreducible form of $m/n$. Now, remarking that
\[\left(\frac{\alpha^k}{\beta^l}\right)^m=1,\]
that is $\alpha^k/\beta^l$ is a root of unity of $L$, compute the set of roots of unity of $L$, and compare it to $\alpha^k/\beta^l$. If none of them is equal to $\alpha^k/\beta^l$, we can conclude that $\alpha$ and $\beta$ are multiplicatively independent.

This method can be extended to show that $\alpha^m\beta^n\notin\mathbb{Q}^{\times}$, for all couples of integers $(m,n)\neq (0,0)$; this is precisely the purpose of Theorem \ref{th:main-2}. Indeed, if $\alpha^m\beta^n\in\mathbb{Q}^{\times}$, then for all automorphism $\sigma$ of the Galois group $\Gal(L/\mathbb{Q})$, we have
\[\left(\frac{\alpha}{\alpha^{\sigma}}\right)^m=\left(\frac{\beta^{\sigma}}{\beta}\right)^n,\]
which means that $\alpha/\alpha^{\sigma}$ and $\beta/\beta^{\sigma}$ are multiplicatively dependent. We may assume here that $L$ contains all the conjugates of $\alpha$ and $\beta$ over $\mathbb{Q}$. Therefore, all we have to do is to find a convenient automorphism $\sigma$ contradicting the latter assertion by applying the previous algorithm.

\section{Proof of Theorem \ref{th:main}}
\label{sec:th_main}

\subsection{Reduction}
\label{ssec:red}

Let $j(\tau),j(\tau')$ be two distinct singular moduli of respective discriminants $\Delta$ and $\Delta'$, and $m,n$ two positive integers. As in Theorem \ref{th:main}, we assume there exist $A,B,C\in\mathbb{Q}^{\times}$ such that
\begin{equation}
\label{eq:main}
Aj(\tau)^m+Bj(\tau')^n=C.
\end{equation}
Dividing by $C$, we can choose $C=1$ in order to simplify notations.

Equation \eqref{eq:main} implies that $\mathbb{Q}(j(\tau)^m)=\mathbb{Q}(j(\tau')^n)$. By Lemma \ref{lemme:power-j}, we deduce that $\mathbb{Q}(j(\tau))=\mathbb{Q}(j(\tau'))$, and in particular, that the discriminants $\Delta$ and $\Delta'$ have the same class number $h=h(\Delta)=h(\Delta')$. Moreover, the Galois orbit of $(j(\tau),j(\tau'))$ over $\mathbb{Q}$ has exactly $h$ elements, and each conjugate of $j(\tau)$ occurs exactly once as the first coordinate of a point in the orbit, just as each conjugate of $j(\tau')$ occurs exactly once as the second coordinate. Frow now on, we assume that $h\geq 3$.

Using a similar reasoning to the proof of Lemma 5.2 of \cite{Al15}, we can easily show that the CM-point $(j(\tau),j(\tau'))$ is conjugate over $\mathbb{Q}$ to a point whose both coordinates are the dominant $j$-values of respective discriminants $\Delta$ and $\Delta'$. Indeed, let $(j(\tau_1),j(\tau_1'))$ be the conjugate of $(j(\tau),j(\tau'))$ with $j(\tau_1)$ dominant. If $j(\tau_1')$ is not dominant, there exist a second conjugate $(j(\tau_2),j(\tau_2'))$ with $j(\tau_2')$ dominant. Eventually, since $h\geq 3$, let $(j(\tau_3),j(\tau_3'))$ be a third conjugate with $j(\tau_3),j(\tau_3')$ both not dominant. All three points $(j(\tau_i)^m,j(\tau_i')^n)$, $i\in\{1,2,3\}$, are collinear, so that
\[\left|
\begin{matrix}
1 & j(\tau_1)^m & j(\tau_1')^n\\
1 & j(\tau_2)^m & j(\tau_2')^n\\
1 & j(\tau_3)^m & j(\tau_3')^n
\end{matrix}
\right|=0.\]
The determinant above is a sum of $6$ terms: the “dominant term” $j(\tau_1)^mj(\tau_2')^n$ and $5$ other terms, each of them being at most $0.1|j(\tau_1)^mj(\tau_2')^n|$ in absolute value, due to lemma \ref{lemme:j-dom}. Hence, the determinant cannot vanish, a contradiction. Consequently, we can assume, without loss of generality, that $j(\tau)$ and $j(\tau')$ are both dominant.

On the one hand, if $\mathbb{Q}(\tau)\neq\mathbb{Q}(\tau')$, then all possible couples $(\Delta,\Delta')$ are given by Table 4.1 of \cite{Al15}, as a consequence of \cite[Corollary 4.2]{Al15}. We will see later in Section \ref{subsec:Qtau_eq} how to eliminate this case.

On the other hand, if $\mathbb{Q}(\tau)=\mathbb{Q}(\tau')$, then by Proposition 4.3 of \cite{Al15}, we have $\Delta\in\{\Delta',4\Delta',\Delta'/4\}$. If $\Delta=\Delta'$, then $j(\tau)$ and $j(\tau')$ are conjugate over $\mathbb{Q}$; however, since they are both dominant, this would imply that $j(\tau)=j(\tau')$, which is excluded. Therefore, we can assume that $\Delta=4\Delta'$. As $j(\tau)$ and $j(\tau')$ are both dominant, we may choose
\[\tau=\frac{-r_4(4\Delta')+\sqrt{4\Delta'}}{2}=\sqrt{\Delta'}\quad\text{and}\quad\tau'=\frac{-r_4(\Delta')+\sqrt{\Delta'}}{2}.\]
Hence, we have $\tau'=\frac{1}{2}\gamma\tau$, with
\[\gamma=\left(
\begin{matrix}
1 & -r_4(\Delta')\\
0 & 1
\end{matrix}
\right)\in\SL(2,\mathbb{Z}).\]
Hence, the CM-point $(j(\tau),j(\tau'))$ belongs to the modular curve $Y_0(2)$, and so do all of its conjugates over $\mathbb{Q}$. This observation will be crucial for our proof.

\subsection{The case $\mathbb{Q}(\tau)=\mathbb{Q}(\tau')$}

In this section, we consider the case $\mathbb{Q}(\tau)=\mathbb{Q}(\tau')$. Before we procede, we shall summarize the different assumptions we will next work with:
\begin{gather}
\label{eq:assum-colin} Aj(\tau)^m+Bj(\tau')^n=1,\\
\label{eq:assum-hsup3} h\geq 3,\\
\label{eq:assum-delta} \Delta=4\Delta',\\
\label{eq:assum-jdom} j(\tau)\text{ and }j(\tau')\text{ are both dominant}.
\end{gather}

Besides, we will freely apply the following properties of the principal complex logarithm:
\begin{gather*}
\forall M\in\,]0,1[,\,\forall z\in\mathbb{C},\,|z-1|\leq M\implies\lvert\log z|\leq\frac{|z-1|}{1-M},\\
\forall M\in\,]0,1[,\,\forall z\in\mathbb{C},\,|z+1|\leq M\implies\lvert\log z-i\pi|\leq\frac{|z+1|}{1-M},
\end{gather*}
which are mere applications of the mean value Theorem. This last property is also useful:
\[\forall s,A>1,\,s\log s<A\implies s<\left(1+\frac{1}{\mathrm{e}}\right)\frac{A}{\log A}.\]

\subsubsection{Eliminating big discriminants}

The first part of the proof is to exhibit a contradiction to the previous assumptions for every discriminant $\Delta$ large enough. That is what we call “eliminating” big discriminants. Afterwards, it only remains finitely many discriminants to deal with.

\sloppypar{The idea is the following. If $(x_1,y_1),(x_2,y_2),(x_3,y_3)$ are three conjugates over $\mathbb{Q}$ of the CM-point $(j(\tau),j(\tau'))$, then the points $(x_i^m,y_i^n)$, $i\in\{1,2,3\}$, all belong to the line of equation $Ax+By=1$. Using the collinearity of these points allows us to bound the quantity $m/n$. Producing different bounds with different set of collinear points leads to a contradiction.}

That is why we need at first to exhibit some explicit conjugates of the CM-point $(j(\tau),j(\tau'))$. Denote $(x_1,y_1)=(j(\tau),j(\tau'))$. Since $\Delta=4\Delta'$, we have $\Delta'\equiv 1\bmod 8$ by the “class number formula”, see \cite[section 3.2.2]{Yu15}. Then, by Lemma \ref{lemme:T_Delta}, assuming that $|\Delta'|\geq 239$, there exist $3$ other conjugates $(x_2,y_2),(x_3,y_3),(x_4,y_4)$ of the following form:

\begin{table}[H]

\caption{Known conjugates of $(j(\tau),j(\tau'))$ over $\mathbb{Q}$, with unknown values of $a_2,a_3,a_4$, and $|\Delta'|\geq 239$}

\[
\begin{array}{r|c|c}
i & x_i & y_i\\
\hline
1 & j(\sqrt{\Delta'}) & j\left(\frac{-1+\sqrt{\Delta'}}{2}\right)\\
2 & j\left(\frac{-b_2}{2a_2}+\frac{\sqrt{\Delta'}}{a_2}\right) & j\left(\frac{-1+\sqrt{\Delta'}}4\right)\\
3 & j\left(\frac{-b_3}{2a_3}+\frac{\sqrt{\Delta'}}{a_3}\right) & j\left(\frac{-b_3'+\sqrt{\Delta'}}{8}\right)\\
4 & j\left(\frac{-b_4}{2a_4}+\frac{\sqrt{\Delta'}}{a_4}\right) & j\left(\frac{-b_4'+\sqrt{\Delta'}}{16}\right).
\end{array}
\]

\end{table}
Now, we would like to determine the values of $a_2,a_3,a_4$. According to Lemma \ref{lemme:Y02}, if $\im((-b_4'+\sqrt{\Delta'})/16)\geq 2$, i.e. $|\Delta'|\geq 1024$, then $a_i/2^i\in\{2,1/2\}$.

\begin{itemize}

\item For $i=2$, that means $a_2\in\{8,1\}$, the case $a_2=1$ being excluded since $x_2$ is not dominant. Therefore, $a_2=8$.

\item Similarly, for $i=3$, we have $a_3\in\{16,2\}$. However, since $\Delta\equiv 4\bmod 16$, there is no triple $(a_3,b_3,c_3)$ in $T_{\Delta}$ with $a_3=2$, so that $a_3=16$.

\item Finally, for $i=4$, we have $a_4\in\{32,4\}$. The two conjugates of $j(\tau)$ corresponding to the triples $(a_4,b_4,c_4)$ of $T_{\Delta}$ with $a_4=4$ are already associated to $y_2$ and $\bar{y_2}$, so that $a_4=32$.

\end{itemize}
Hence, we have complete expressions for the previous conjugates when $|\Delta'|\geq 1024$.

\begin{table}[!h]

\caption{Known conjugates of $(j(\tau),j(\tau'))$ over $\mathbb{Q}$, with $|\Delta'|\geq 1024$}

\[
\begin{array}{r|c|c}
i & a_i^{-1} & a_i'^{-1}\\
\hline
1 & 1 & 1/2\\
2 & 1/8 & 1/4\\
3 & 1/16 & 1/8\\
4 & 1/32 & 1/16.
\end{array}
\]

\footnotesize

Explanations: for each $i$, the corresponding conjugates $x_i$ and $y_i$ are of the form $x_i=j\left(\ast+\frac{\sqrt{\Delta'}}{a_i}\right)$ and $y_i=j\left(\ast+\frac{\sqrt{\Delta'}}{a_i'}\right)$, where $\ast$ is a real value.

\label{table:conj}

\end{table}

\noindent Remark that we have not determined the values of $b_2,b_3,b_4,b_3',b_4'$, but we do not need to know them to estimate all of the $x_i$'s and $y_i$'s, since the magnitude of $|j(\tau)q(\tau)|$, for $\tau\in\mathcal{D}$, only depends on $\im\tau$.

For this section only, we assume that $|\Delta'|\geq 1024$, and thus table \ref{table:conj} above gives us four explicit conjugates of the CM-point $(j(\tau),j(\tau'))$ over $\mathbb{Q}$.

Using Lemma \ref{lemme:estim-j-bis}, we have the estimations below:
\begin{equation}
\label{eq:estimxiyi}
\left|\frac{x_i}{x_j}\right|=\mathrm{e}^{2\pi|\Delta'|^{1/2}(a_i^{-1}-a_j^{-1})+O(1)},\quad\left|\frac{y_i}{y_j}\right|=\mathrm{e}^{2\pi|\Delta'|^{1/2}(a_i'^{-1}-a_j'^{-1})+O(1)},\quad 1\leq i<j\leq 4,\\
\end{equation}
with $|O(1)|\leq 1$ in each estimation.

By collinearity of the points $(x_i^m,y_i^n)$, $i\in\{1,2,3\}$, we have
\begin{equation}
\label{eq:colin-det}
\begin{vmatrix}
1 & x_1^m & y_1^n\\
1 & x_2^m & y_2^n\\
1 & x_3^m & y_3^n
\end{vmatrix}=0,
\end{equation}
which can be rewritten as
\begin{equation}
\label{eq:colin}
\left(\frac{x_1}{x_2}\right)^m\left(\frac{y_1}{y_2}\right)^{-n}-1=\frac{\left(\frac{y_3}{y_2}\right)^n+\left(\frac{x_3}{x_1}\right)^m-\left(\frac{y_3}{y_1}\right)^n-\left(\frac{x_3}{x_2}\right)^m}{1-\left(\frac{y_3}{y_2}\right)^n-\left(\frac{x_3}{x_1}\right)^m}.
\end{equation}
By estimations \eqref{eq:estimxiyi}, we have
\[\left|\left(\frac{y_3}{y_2}\right)^n+\left(\frac{x_3}{x_1}\right)^m\right|\leq\mathrm{e}^{-\pi\sqrt{1024}/4+1}+\mathrm{e}^{-15\pi\sqrt{1024}/8+1}\leq 0.001,\]
thus
\begin{equation}
\label{eq:maj-mn}
\left|\left(\frac{x_1}{x_2}\right)^m\left(\frac{y_1}{y_2}\right)^{-n}-1\right|\leq 2.004\left(\mathrm{e}^{-n\pi|\Delta'|^{1/2}/4+n}+\mathrm{e}^{-m\pi|\Delta'|^{1/2}/8+m}\right)\leq\frac{1}{2}.
\end{equation}

From \eqref{eq:estimxiyi} and \eqref{eq:maj-mn}, we can deduce a bound on the quantity $m/n$. Indeed, we have
\[\left|\left(\frac{7m}{4}-\frac{n}{2}\right)\pi|\Delta'|^{1/2}+mO(1)+nO(1)\right|\leq 4.008\left(\mathrm{e}^{-n\pi|\Delta'|^{1/2}/4+n}+\mathrm{e}^{-m\pi|\Delta'|^{1/2}/8+m}\right)\leq 0.001,
\]
then
\[\left|\frac{7m}{4}-\frac{n}{2}\right|\leq\frac{0.001+m+n}{\pi|\Delta'|^{1/2}}\leq 0.001+0.01(m+n),\]
and eventually
\begin{equation}
\label{eq:bornmn1}
0.279\leq\frac{m}{n}\leq 0.294.
\end{equation}

The previous calculations can be repeated, exploiting this time the collinearity of the points $(x_i^m,y_i^n)$, $i\in\{1,3,4\}$: that is, replacing $(x_2,y_2)$ by $(x_3,y_3)$, $(x_3,y_3)$ by $(x_4,y_4)$, and using estimations \eqref{eq:estimxiyi} for these corresponding conjugates. This leads to
\[\left|\left(\frac{x_1}{x_3}\right)^m\left(\frac{y_1}{y_3}\right)^{-n}-1\right|\leq 2.004\left(\mathrm{e}^{-n\pi|\Delta'|^{1/2}/8+n}+\mathrm{e}^{-m\pi|\Delta'|^{1/2}/16+m}\right),\]
then
\[\left|\frac{15m}{8}-\frac{3n}{4}\right|\leq 0.001+0.01(m+n),\]
and eventually
\begin{equation}
\label{eq:bornmn2}
0.392\leq\frac{m}{n}\leq 0.409.
\end{equation}

The two inequalities \eqref{eq:bornmn1} and \eqref{eq:bornmn2} contradict each other. Therefore, this eliminates all discriminants $|\Delta'|\geq 1024$.

\subsubsection{Eliminating small discriminants}
\label{sssec:small_disc}

It only remains to eliminate discriminants $\Delta'$ satisfying $|\Delta'|<1024$ and $h(\Delta')=h(4\Delta')\geq 3$. The list of such discriminants can be easily computed using \textsf{PARI}.

The previous estimations are not accurate enough to repeat the same calculations when $\Delta'$ is small. Nevertheless, for a fixed discriminant $\Delta'$, we can compute numerically all the conjugates of the CM-point $(j(\tau),j(\tau'))$ over $\mathbb{Q}$, and then procede to explicit estimations. Another difficulty arises from the number of conjugates of $(j(\tau),j(\tau'))$. Indeed, when $\Delta'$ is small, the existence of four conjugates $(x_i,y_i)$, $i\in\{1,2,3,4\}$, for which the absolute values of the $x_i$'s and $y_i$'s are respectively all distinct, is not guaranteed. The case corresponding to the existence of only three such conjugates requires a special treatment. In this case, we can bound both $m$ and $n$ thanks to known estimations of linear forms in the logarithms of algebraic numbers, and using some theory of continued fractions leads to a few values of $m$ and $n$ to examine.

In this subsection, we assume that $|\Delta'|<1024$. Denote $(x_1,y_1),\dots,(x_r,y_r)$ a maximal system of conjugates of the CM-point $(j(\tau),j(\tau'))$ over $\mathbb{Q}$, keeping only one couple for each pair of complex conjugates couples. It can be checked by a simple calculation in \textsf{PARI} that, for all discriminants $|\Delta'|<1024$, the absolute values of the $x_i$'s and $y_i$'s are respectively all distinct. We may still choose $(x_1,y_1)=(j(\tau),j(\tau'))$, the conjugate corresponding to the dominant $j$-values $j(\tau)$ and $j(\tau')$ of respective discriminants $\Delta$ and $\Delta'$. We may also index the $x_i$'s such that $|x_1|>|x_2|>\dots>|x_r|$. Moreover, we consider that $r\geq 3$. It turns out that, under the assumptions \eqref{eq:assum-hsup3} and \eqref{eq:assum-delta}, this condition is satisfied except for the discriminants $\Delta'\in\{-23,-31\}$.

Let $i,j\in\{2,\dots,r\}$, $i<j$. We can apply again the previous method to bound the quantity $m/n$, using the collinearity of the points $(x_k,y_k)$, $k\in\{1,i,j\}$. The only detail we have to pay attention to is the comparison between $|y_i|$ and $|y_j|$. From the identity
\[
\begin{vmatrix}
1 & x_1^m & y_1^n\\
1 & x_i^m & y_i^n\\
1 & x_j^m & y_j^n
\end{vmatrix}=0,
\]
we write the two following equations:
\begin{gather}
\label{eq:colingen}
\left(\frac{x_1}{x_i}\right)^m\left(\frac{y_1}{y_i}\right)^{-n}-1=\frac{\left(\frac{y_j}{y_i}\right)^n+\left(\frac{x_j}{x_1}\right)^m-\left(\frac{y_j}{y_1}\right)^n-\left(\frac{x_j}{x_i}\right)^m}{1-\left(\frac{y_j}{y_i}\right)^n-\left(\frac{x_j}{x_1}\right)^m},\\
\label{eq:colingen-rev}
\left(\frac{x_1}{x_i}\right)^m\left(\frac{y_1}{y_j}\right)^{-n}+1=\frac{\left(\frac{x_j}{x_i}\right)^m\left(-1+\left(\frac{y_i}{y_1}\right)^n\right)+\left(\frac{y_i}{y_j}\right)^n+\left(\frac{x_i}{x_1}\right)^m}{-1+\left(\frac{y_i}{y_j}\right)^n+\left(\frac{x_i}{x_1}\right)^m}.
\end{gather}
If follows that:

\begin{itemize}

\item if $|y_i|>|y_j|$, then by \eqref{eq:colingen}, $|(x_1/x_i)^m(y_1/y_i)^{-n}-1|$ is bounded by an effective constant $M<1$ given by
\[M=\frac{|y_j/y_i|+|x_j/x_1|+|y_j/y_1|+|x_j/x_i|}{1-|y_j/y_i|-|x_j/x_1|},\]
and we get
\begin{equation}
\label{eq:bornmn1-eff}
\frac{-M+(1-M)\log|y_1/y_i|}{(1-M)\log|x_1/x_i|}\leq\frac{m}{n}\leq\frac{M+(1-M)\log|y_1/y_i|}{(1-M)\log|x_1/x_i|};
\end{equation}

\item if $|y_i|<|y_j|$, then by \eqref{eq:colingen-rev}, $|(x_1/x_i)^m(y_1/y_j)^{-n}+1|$ is bounded by an effective constant $M<1$ given by
\[M=\frac{|x_j/x_i|(1+|y_i/y_1|)+|y_i/y_j|+|x_i/x_1|}{1-|y_i/y_j|-|x_i/x_1|},\]
and we get
\begin{equation}
\label{eq:bornmn2-eff}
\frac{-M+(1-M)\log|y_1/y_j|}{(1-M)\log|x_1/x_i|}\leq\frac{m}{n}\leq\frac{M+(1-M)\log|y_1/y_j|}{(1-M)\log|x_1/x_i|}.
\end{equation}

\end{itemize}

If $r\geq 4$, we can thus compute several bounds on the quantity $m/n$, one for each couple $(i,j)\in\{2,\dots,r\}^2$ with $i<j$, and we just have to find two bounds contradicting each other (that is two disjoint intervals for $m/n$) to eliminate this possibility. It is a simple verification in \textsf{PARI}.

From now on, we consider that $r=3$, and we have $h\leq 5$. The corresponding discriminants $\Delta'$ are $-39,-47,-55,-63,-79,-103,-127$.

Denote $\alpha=x_1/x_2$ and $\beta=y_1/y_2$ if $|y_2|>|y_3|$, else $\beta=y_1/y_3$. Recall that there exist effective constants $M<1$ and $c_1<c_2<1$ such that
\begin{equation}
\label{eq:estimab}
|\alpha^m\beta^{-n}+(-1)^{\varepsilon}|\leq M,
\end{equation}
with $\varepsilon=1$ if $|y_2|>|y_3|$, otherwise $\varepsilon=0$, and
\begin{equation}
\label{eq:bornmn}
c_1\leq m/n\leq c_2.
\end{equation}
In particular, $m\leq n$.

From \eqref{eq:colingen} and \eqref{eq:colingen-rev}, we can compute two constants $c_3>1$ and $c_4<1$ such that
\begin{equation}
\label{eq:estimab-prec}
|\alpha^m\beta^{-n}+(-1)^{\varepsilon}|\leq c_3.c_4^n,
\end{equation}
Therefore,
\begin{equation}
\label{eq:maxlambda}
|m\log\alpha-n\log\beta+(2k-1+\varepsilon)i\pi|\leq c_3'.c_4^n,
\end{equation}
with $c_3'=c_3/(1-M)>1$. Here, $\log$ denotes the principal complex logarithm (defined on $\mathbb{C}\setminus\mathbb{R}^-$), and the integer $k$ must satisfy $|2k|\leq m+n$.

Combining inequation \eqref{eq:maxlambda} with Theorem \ref{th:matveev} allows us to bound $n$ as follows. Denote $\Lambda=m\log\alpha-n\log\beta+(2k-1+\varepsilon)i\pi$, and notice that $i\pi=\log(-1)$. Before we procede, we have to verify that $\Lambda\neq 0$. If $\Lambda=0$, then either
\[\left(\frac{x_1}{x_2}\right)^m=\left(\frac{y_1}{y_2}\right)^n\]
and $x_1/x_2$ and $y_1/y_2$ are multiplicatively dependent if $|y_2|>|y_3|$, or
\[\left(\frac{x_1}{x_2}\right)^m=-\left(\frac{y_1}{y_3}\right)^n\]
and $x_1/x_2$ and $y_1/y_3$ are multiplicatively dependent if $|y_2|<|y_3|$. Both possibilites can be eliminated for all the discriminants mentioned above by using the algorithm described in Subsection \ref{ssec:mult_ind}.

All we need to do now is to estimate the different parameters occuring in \eqref{eq:matveev}.

\begin{itemize}

\item Remarking that $\mathbb{Q}(\alpha,\beta)\subset K(x_1,y_1)=K(x_1)$, where $K=\mathbb{Q}(\tau)=\mathbb{Q}(\tau')$, we have
\[d=[\mathbb{Q}(\alpha,\beta,-1):\mathbb{Q}]\leq[K(x_1):\mathbb{Q}]=2h\leq 10.\]

\item By Lemma \ref{lemme:height-j}, we have
\[\mathrm{h}(\alpha)\leq 2\mathrm{h}(x_1)\leq 9|\Delta|^{1/2}=18|\Delta'|^{1/2}.\]
Besides,
\[\lvert\log\alpha|/d\leq\lvert\log|\alpha||/d+\pi/d\leq\mathrm{h}(\alpha)+\pi/d\leq 19|\Delta'|^{1/2}.\]
Hence we can take $A_1=19|\Delta'|^{1/2}$. By the same way, we can take $A_2=10|\Delta'|^{1/2}$, and $A_3=1$.

\item At last, $H=\max\{m,n,|2k-1+\varepsilon|\}\leq m+n+1\leq n(1+c_2)+1\leq 3n$.

\end{itemize}

\noindent Therefore,
\begin{equation}
\label{eq:minlambda}
|\Lambda|>\mathrm{e}^{-1671257674219520h^5|\Delta'|\log(2\mathrm{e}h)\log(3\mathrm{e}n)}.
\end{equation}

Finally, by inequations \eqref{eq:maxlambda} and \eqref{eq:minlambda}, we get
\begin{equation}
\label{eq:maj-n/logn}
\frac{n}{\log(3\mathrm{e}n)}<\frac{-1671257674219520h^5|\Delta'|\log(2\mathrm{e}h)-\log c_3'}{\log c_4}=c_5,
\end{equation}
and then
\begin{equation}
\label{eq:maj-n}
n<\left(1+\frac{1}{\mathrm{e}}\right)c_5\log(3(1+\mathrm{e})c_5)=c_6.
\end{equation}

In practice, the bound we obtain on $n$ is too large to examine all possible $m$ and $n$. That is why we connect the fraction $m/n$ to the continued fraction expansion of some real number.

Considering again estimation \eqref{eq:estimab-prec}, we have
\begin{equation}
\label{eq:estimtheta}
\left|\theta-\frac{m}{n}\right|\leq\frac{c_3'.c_4^n}{n\log|\alpha|},
\end{equation}
with $\theta=\log|\beta|/\log|\alpha|$. Hence we obtain a sharp approximation of the real number $\theta$ by the rational number $m/n$. By \cite[Theorem 5C]{Sc80}, two situations can occur.

\begin{itemize}

\item If $|\theta-m/n|\leq 1/(2n^2)$, then writing $m/n$ into irreducible form $p/q$, the fraction $p/q$ is a convergent to $\theta$.

\item Otherwise, we have
\[n\leq\frac{\log(2c_3n/\log|\alpha|)}{\log(c_4^{-1})}\leq\frac{\log(2c_3c_6/\log|\alpha|)}{\log(c_4^{-1})}=c_6',\]
the last inequality coming from \eqref{eq:maj-n}. Notice that this process can be repeated until we find an upper bound for $n$ less than $1$: indeed, if $c_6'>1$, we have
\[n\leq\frac{\log(2c_3c_6')/\log|\alpha|}{\log(c_4^{-1})}=c_6'',\]
and so on, each new bound being smaller than the previous one. Hence, this situation does not occur.

%In this case, the bound $c_6'$ we obtain on $n$ is much lower than the previous one $c_6$. It is then possible to enumerate all possible couples $(m,n)$ (recall that $m\leq c_2n$ by \eqref{eq:bornmn}). For each of these couples, we can numerically check that equation \eqref{eq:colin-det} is not satisfied.

\end{itemize}

Consequently, all possible couples $(m,n)$ are given by the multiples of the convergents $p/q$ of the continued fraction expansion of $\theta$ and the bound \eqref{eq:maj-n} on $n$. Finally, remarking that the right term in the inequality \eqref{eq:estimtheta} converges much faster to $0$ than its left term as $n$ grows, leads us to the following verification in \textsf{PARI}. For each convergent $p/q$, we check that
\begin{equation}
\label{eq:neg-estimtheta}
\log|\theta q-p|>q\log(c_4')+\log\left(\frac{c_3'}{\log|\alpha|}\right),
\end{equation}
which contradicts directly \eqref{eq:estimtheta} for the given couple $(p,q)$. Since the right term of \eqref{eq:neg-estimtheta} decreases as $q$ grows, then if \eqref{eq:neg-estimtheta} is satisfied for $(p,q)$, it is also satisfied for all multiples $(m,n)$ of $(p,q)$. Therefore, we only have to check that \eqref{eq:neg-estimtheta} is satisfied for all convergents $p/q$ up to $q<c_6$, the bound we obtained in \eqref{eq:maj-n}.

%\begin{remarque}
%
%Computing the exact convergents $p/q$ of the continued fraction expansion of $\theta$ up to $q<c_6$ requires a precision good enough on the real number $\theta$, so we might increase the precision in our calculations in \textsf{PARI}. The function \textsf{cont\_frac} computes all possible convergents with the given precision, see \cite[Algorithm 1.3.13]{Co96} for a description of the algorithm used.
%
%\end{remarque}

\subsection{The case $\mathbb{Q}(\tau)\neq\mathbb{Q}(\tau')$}
\label{subsec:Qtau_eq}

To conclude the proof, we have to show that the case $\mathbb{Q}(\tau)\neq\mathbb{Q}(\tau')$ cannot occur. Assume that $\mathbb{Q}(\tau)\neq\mathbb{Q}(\tau')$. Here we still work under assumptions \eqref{eq:assum-colin}, \eqref{eq:assum-hsup3} and \eqref{eq:assum-jdom}. The only difference is that all possible pairs $\{\Delta,\Delta'\}$ are given by Table 4.1 of \cite{Al15}, reproduced below.

\begin{table}[H]

\caption{Fields presented as $\mathbb{Q}(j(\tau))$ and $\mathbb{Q}(j(\tau'))$ with $\mathbb{Q}(\tau)\neq\mathbb{Q}(\tau')$}

\[
\begin{array}{l|l|l}
\text{Field $L$} & [L:\mathbb{Q}] & \Delta\\
\hline
\mathbb{Q} & 1 & -3, -4, -7, -8, -11, -12, -16, -19, -27, -28, -43,\\ 
& & -67, -163\\
\hline
\mathbb{Q}(\sqrt{2}) & 2 & -24, -32, -64, -88\\
\mathbb{Q}(\sqrt{3}) & 2 & -36, -48\\
\mathbb{Q}(\sqrt{5}) & 2 & -15, -20, -35, -40, -60, -75, -100, -115, -235\\
\mathbb{Q}(\sqrt{13}) & 2 & -52, -91, -403\\
\mathbb{Q}(\sqrt{17}) & 2 & -51, -187\\
\hline
\mathbb{Q}(\sqrt{2},\sqrt{3}) & 4 & -96, -192, -288\\
\mathbb{Q}(\sqrt{3},\sqrt{5}) & 4 & -180, -240\\
\mathbb{Q}(\sqrt{5},\sqrt{13}) & 4 & -195, -520, -715\\
\mathbb{Q}(\sqrt{2},\sqrt{5}) & 4 & -120, -160, -280, -760\\
\mathbb{Q}(\sqrt{5},\sqrt{17}) & 4 & -340, -595\\
\hline
\mathbb{Q}(\sqrt{2},\sqrt{3},\sqrt{5}) & 8 & -480, -960
\end{array}
\]

\label{table:double}

\end{table}

The idea we developed in \ref{sssec:small_disc} can be used again here to eliminate all possible pairs $\{\Delta,\Delta'\}$ with $\Delta\neq\Delta'$, $h(\Delta)=h(\Delta')\geq 3$ and $\mathbb{Q}(j(\tau))=\mathbb{Q}(j(\tau'))$: for each such pair, we determine three conjugates $(x_1,y_1),(x_2,y_2),(x_3,y_3)$ of the CM-point $(j(\tau),j(\tau'))$ by identifying first the dominant $j$-values of respective discriminants $\Delta$ and $\Delta'$; then, Theorem \ref{th:matveev} allows to bound $n$ as previously, except that \eqref{eq:minlambda} is replaced by
\[|\Lambda|>\mathrm{e}^{-2748779069440h^5|\Delta|^{1/2}|\Delta'|^{1/2}\log(\mathrm{e}h)\log(3\mathrm{e}n)},\]
and the constants $c_5$ and $c_6$ given by equations \eqref{eq:maj-n/logn} and \eqref{eq:maj-n} respectively change consequently; finally, all possible couples $(m,n)$ are given by the continued fraction expansion of some real $\theta$ by \eqref{eq:estimtheta}. A similar calculation in \textsf{PARI} leads to eliminating all remaining pairs $\{\Delta,\Delta'\}$, which finally demonstrates Theorem \ref{th:main}.

\section{Proof of Theorem \ref{th:main-2}}

\subsection{Reduction}
\label{ssec:red-2}

Let $j(\tau),j(\tau')$ be two singular moduli of respective discriminants $\Delta$ and $\Delta'$, and $m,n$ two non-zero integers. Assume that $j(\tau)^mj(\tau')^n=A\in\mathbb{Q}^{\times}$.

First of all, we clearly have $\mathbb{Q}(j(\tau)^m)=\mathbb{Q}(j(\tau')^n)$, hence $\mathbb{Q}(j(\tau))=\mathbb{Q}(j(\tau'))$ by Lemma \ref{lemme:power-j} again. We set $h=h(\Delta)=h(\Delta')$. When $h=1$, we are obviously in the “rational case” of Theorem \ref{th:main-2}. Moreover, if $j(\tau)=j(\tau')$, then $j(\tau)^{m+n}\in\mathbb{Q}^{\times}$, and either $m+n=0$, that is the “equality case”, or $m+n\neq 0$ and $j(\tau)\in\mathbb{Q}^{\times}$ since $\mathbb{Q}(j(\tau)^{m+n})=\mathbb{Q}(j(\tau))$.

From now on, we assume that $h\geq 2$.

The method we described in Subsection \ref{ssec:mult_ind} is useful to “eliminate” some particular pairs of discriminants $\{\Delta,\Delta'\}$; here, “eliminating” means showing that for any choice of singular moduli $j(\tau)$ and $j(\tau')$ of respective discriminants $\Delta$ and $\Delta'$, we have $j(\tau)^mj(\tau')^n\notin\mathbb{Q}^{\times}$. It is equivalent to fix two particular singular moduli $j(\tau)$ and $j(\tau')$ of respective discriminants $\Delta$ and $\Delta'$, and to show that for all automorphism $\sigma$ of the Galois group $\Gal(L/\mathbb{Q})$, we have $j(\tau)^m(j(\tau')^{\sigma})^n\notin\mathbb{Q}^{\times}$, where $L$ is a field containing all the conjugates of $j(\tau)$ and $j(\tau')$ over $\mathbb{Q}$; we can choose $L=\mathbb{Q}(\tau,j(\tau))$ if $\mathbb{Q}(\tau)=\mathbb{Q}(\tau')$, else $L=\mathbb{Q}(j(\tau))$ and $L$ is one of the field given by Table \ref{table:double}. We explained in Subsection \ref{ssec:mult_ind} how to treat this problem.

In particular, we were able to eliminate all pairs of discriminants $\{\Delta,\Delta'\}$ given by Table \ref{table:double} with $h\geq 3$ (corresponding to the case $\mathbb{Q}(\tau)\neq\mathbb{Q}(\tau')$) and all pairs with $\Delta\neq\Delta'$ and $h=2$. In the case $\Delta=\Delta'$ and $h=2$, $j(\tau)$ and $j(\tau')$ are conjugate over $\mathbb{Q}$, and as $j(\tau)\neq j(\tau')$, they must satisfy
\[\left(\frac{j(\tau)}{j(\tau')}\right)^{m-n}=1.\]
Since $|j(\tau)|\neq |j(\tau')|$ (one of these two singular moduli is the dominant $j$-value of discriminant $\Delta$), we deduce $m=n$. This corresponds to the “quadratic case” of Theorem \ref{th:main-2}.

Therefore, in the sequel, we assume that $h\geq 3$ and $\mathbb{Q}(\tau)=\mathbb{Q}(\tau')$. Recall that we have then $\Delta/\Delta'\in\{1,4,1/4\}$.

The proof splits into two different parts: we consider separately the cases $mn<0$ and $mn>0$. In both cases, our aim is to reduce the problem to finitely many pairs of discriminants $\{\Delta,\Delta'\}$ to examine, and apply the same method as before to eliminate the remaining cases.

%We denote by $x_1,\dots,x_h$ and $y_1,\dots,y_h$ the distinct conjugates of $j(\tau)$ and $j(\tau')$ respectively over $\mathbb{Q}$. We also consider a field $L$ containing all the $x_i$'s and $y_i$'s: we can choose $L=\mathbb{Q}(\tau,j(\tau))$ if $\mathbb{Q}(\tau)=\mathbb{Q}(\tau')$, else $L=\mathbb{Q}(j(\tau))$ and $L$ is one of the field given by Table... . We want to eliminate the possibility that $j(\tau)^mj(\tau')^n\in\mathbb{Q}^{\times}$ for any choice of singular moduli $j(\tau)$ and $j(\tau')$ of respective discriminants $\Delta$ and $\Delta'$. This is equivalent to showing that $x_i^my_j^n\notin\mathbb{Q}^{\times}$, for all $i\in\{1,\dots,h\}$. By conjugating over $\mathbb{Q}$, we may fix $i=1$. Now, setting $\sigma$ an arbitrary automorphism of the Galois group $\Gal(L/\mathbb{Q})$, we have
%\begin{equation}
%\label{eq:mult_ind}
%\left(\frac{x_1}{x_1^{\sigma}}\right)^m=\left(\frac{y_j^{\sigma}}{y_j}\right)^n.
%\end{equation}

\subsection{The case $mn<0$}

We consider the case $mn<0$. Noticing that the CM-point $(j(\tau),j(\tau'))$ belongs either to the curve of equation $x^m-Ay^{-n}=0$ if $m>0$ and $n<0$, or to the curve of equation $Ax^{-m}-y^n=0$ if $m<0$ and $n>0$, we are in the situation of Theorem \ref{th:main} with $C=0$. In particular, the reduction we did in section \ref{ssec:red} is still valid: $(j(\tau),j(\tau'))$ is conjugate over $\mathbb{Q}$ to a couple whose both coordinates are the dominant $j$-values of respective discriminants $\Delta$ and $\Delta'$. In particular, $\Delta\neq\Delta'$, since $j(\tau)\neq j(\tau')$. Thereby, we can assume without loss of generality that $\Delta=4\Delta'$, and we have $(j(\tau),j(\tau'))\in Y_0(2)$.

It follows that, when $|\Delta'|\geq 256$, Table \ref{table:conj} provides us three explicit conjugates $(x_1,y_1),(x_2,y_2),(x_3,y_3)$ of the CM-point $(j(\tau),j(\tau'))$ over $\mathbb{Q}$. These conjugates satisfy
\[\left(\frac{x_1}{x_i}\right)^m=\left(\frac{y_1}{y_i}\right)^{-n},\quad i\in\{2,3\}.\]
Hence, we have
\begin{equation}
\label{eq:m/n-mn<0}
-\frac{m}{n}=\frac{\log|y_1/y_i|}{\log|x_1/x_i|},\quad i\in\{2,3\}.
\end{equation}
Combining estimations \eqref{eq:estimxiyi} with relations \eqref{eq:m/n-mn<0} gives
\begin{equation}
\label{eq:egal-mn}
\frac{1/2+O(1)}{7/4+O(1)}=\frac{3/4+O(1)}{15/8+O(1)},
\end{equation}
with $|O(1)|\leq 0.001$ here. One can thus easily observe that the left term of \eqref{eq:egal-mn} is less than its right term in absolute values: this leads to a contradiction and eliminates all discriminants $|\Delta'|\geq 256$. It only remains finitely many cases to eliminate, that is pairs of discriminants $\{\Delta,\Delta'\}$ satisfying $\Delta=4\Delta'$, $h(\Delta)=h(\Delta')\geq 3$ and $|\Delta'|<256$. We eliminated them with the usual method described in \ref{ssec:mult_ind} and \ref{ssec:red-2}.

\subsection{The case $mn>0$}

Now we consider the case $mn>0$. Without loss of generality, since $j(\tau)^{-m}j(\tau')^{-n}\in\mathbb{Q}^{\times}$, we may assume that $m,n>0$ and additionnally $m\geq n$.

In this subsection, we assume that $h>6$, and additionnally $m\geq n$ without loss of generality. In particular,
\[|\Delta|,|\Delta'|\geq 71.\]

We mainly follow the proof presented in Section 3 of \cite{Yu15}. The idea is to bound $A$ from below and from above in terms of $\Delta$ and $\Delta'$, and see that the two bounds contradict each other in all but finitely many cases. We refer to this article for all the details.

We want to bound $A$ from below. On the one hand, by conjugating over $\mathbb{Q}$, we may assume that $j(\tau)$ is the dominant $j$-value of discriminant $\Delta$, and we deduce from Lemma \ref{lemme:estim-j} that
\[|j(\tau)|\geq\mathrm{e}^{\pi|\Delta|^{1/2}}-2079\geq 0.999\mathrm{e}^{\pi|\Delta|^{1/2}}.\]
On the other hand, \cite[Proposition 2.2]{Yu15} implies that
\[|j(\tau')|\geq\min\{4.4\cdot 10^{-5},\,3500|\Delta'|^{-3}\}.\]
Hence we obtain the following lower estimate of $A$:
\begin{equation}
\label{eq:A-min}
|A|\geq \left(0.999\mathrm{e}^{\pi|\Delta|^{1/2}}\right)^m\left(\min\{4.4\cdot 10^{-5},\,3500|\Delta'|^{-3}\}\right)^n\geq\left(3000\mathrm{e}^{\pi|\Delta|^{1/2}}\min\{10^{-8},|\Delta'|^{-3}\right)^m
\end{equation}
(recall that $m\geq n$ for the last inequality).

Now we want to bound $A$ from above. Every conjugate of the couple $(j(\tau),j(\tau'))$ is of the form
\[\left(j\left(\frac{-b+\sqrt{\Delta}}{2a}\right),j\left(\frac{-b'+\sqrt{\Delta'}}{2a'}\right)\right),\quad(a,b,c)\in T_{\Delta},\quad(a',b',c')\in T_{\Delta'}.\]
The larger enough $a$ and $a'$ are, the sharper the upper estimate we obtain on $A$ is. Since $h>6$ and there are at most two triples in $T_{\Delta}$ (resp. $T_{\Delta'}$) with $a=2$ (resp. $a'=2$), there is a conjugate of $(j(\tau),j(\tau'))$ of the previous form with $a,a'\geq 3$. Using again Lemma \ref{lemme:estim-j}, we deduce
\begin{equation}
\label{eq:A-max}
|A|\leq\left(\mathrm{e}^{\pi|\Delta|^{1/2}/3}+2079\right)^m\left(\mathrm{e}^{\pi|\Delta'|^{1/2}/3}+2079\right)^n\leq\left(1.71\mathrm{e}^{\pi(|\Delta|^{1/2}+|\Delta'|^{1/2})/3}\right)^m.
\end{equation}

Finally, combining \eqref{eq:A-min} and \eqref{eq:A-max} leads to
\begin{equation}
\label{eq:A-minmax}
3000\mathrm{e}^{\pi|\Delta|^{1/2}}\min\{10^{-8},|\Delta'|^{-3}\}\leq 1.71\mathrm{e}^{\pi(|\Delta|^{1/2}+|\Delta'|^{1/2})/3}.
\end{equation}

\begin{itemize}

\item If $\Delta=\Delta'$, inequation \eqref{eq:A-minmax} becomes
\[3000\mathrm{e}^{\pi|\Delta|^{1/2}/3}\min\{10^{-8},|\Delta|^{-3}\}\leq 1.71,\]
and yields $|\Delta|\leq 109$. The only discriminants $\Delta$ such that $h(\Delta)>6$ and $|\Delta|\leq 109$ are $-71$ and $-95$.

\item If $\Delta=4\Delta'$, inequation \eqref{eq:A-minmax} becomes this time
\[3000\mathrm{e}^{\pi|\Delta'|^{1/2}}\min\{10^{-8},|\Delta'|^{-3}\}\leq 1.71,\]
and yields $|\Delta'|\leq 12$, a contradiction.

\item If $\Delta'=4\Delta$, inequality \eqref{eq:A-minmax} does not allow us to bound $\Delta$ and $\Delta'$ as previously. However, the ``class number formula'' implies that $\Delta\equiv 1\bmod 8$, see \cite[section 3.2.2]{Yu15}. Then, according to Lemma \ref{lemme:T_Delta}, there are exactly two triples $(a,b,c)\in T_{\Delta}$ with $a=8$ provided that $|\Delta|\geq 239$, whereas there is no triple $(a',b',c')\in T_{\Delta'}$ with $a'=2$. Consequently, there is a conjugate of $(j(\tau),j(\tau'))$ with $a=8$ and $a'\geq 3$, and we obtain a new upper estimate for $A$:
\[|A|\leq \left(\left(\mathrm{e}^{\pi|\Delta|^{1/2}/8}+2079\right)\left(\mathrm{e}^{2\pi|\Delta|^{1/2}/3}+2079\right)\right)^m.\]
Combining with \eqref{eq:A-min} leads to
\[3000\mathrm{e}^{\pi|\Delta|^{1/2}}\min\{10^{-8},|\Delta|^{-3}/8\}\leq \left(\mathrm{e}^{\pi|\Delta|^{1/2}/8}+2079\right)\left(\mathrm{e}^{2\pi|\Delta|^{1/2}/3}+2079\right)\]
and yields $|\Delta|\leq 310$.

\end{itemize}

It only remains to study finitely many cases, that is pairs of discriminants $\{\Delta,\Delta'\}$ satisfying $h(\Delta)=h(\Delta')$ and one of the three following conditions:
\begin{itemize}

\item $3\leq h(\Delta)\leq 6$ and $\Delta/\Delta'\in\{1,4\}$;

\item $\Delta=\Delta'\in\{-71,-95\}$;

\item $h(\Delta)>6$, $\Delta'=4\Delta$ and $|\Delta|\leq 310$.

\end{itemize}
Again, we eliminated them with the usual method described in Subsections~\ref{ssec:mult_ind} and~\ref{ssec:red-2}.

%===============

%===============
\end{document}